\documentclass[letterpaper,onecolumn,english,aps,showpacs, notitlepage, prl]{revtex4-1}

\usepackage[T1]{fontenc}
\usepackage[latin1]{inputenc}
\usepackage{float}
\usepackage{graphicx}
\usepackage{amsmath,amsfonts,amssymb,amsthm,epsfig,epstopdf,url,array,enumerate,csquotes}

\theoremstyle{definition}
\newtheorem{definition}{Definition}[section]

\newtheorem{theorem}{Theorem}[section]

\newtheorem{lemma}[theorem]{Lemma}

\usepackage{bm}
\usepackage{graphicx,color,psfrag}
\usepackage{color}
\usepackage{babel}
\usepackage{todonotes}
\makeatother
  
\begin{document}




\title{THE ABELIAN DISTRIBUTION}

\author{ANNA LEVINA}

\address{Max Planck Institute for Mathematics in Sciences\\
Kreutzstrasse 21, 04287 Leipzig, Germany\\
Bernstein Center for Computational Neuroscience, G\"{o}ttingen\\
Am Fassberg 12, 37077 G\"{o}ttingen, Germany\\
anna@nld.ds.mpg.de}

\author{J. MICHAEL HERRMANN}

\address{IPAB, School of Informatics, University of Edinburgh\\
10 Crichton St, Edinburgh, EH8 9AB, U.K.\\
j.michael.herrmann@gmail.com}

\begin{abstract}
We define the Abelian distribution and study its basic properties. Abelian distributions arise 
in the context of neural modeling and describe the size of neural avalanches in fully-connected 
integrate-and-fire models of self-organized criticality in neural systems. 
\end{abstract}

\maketitle



\section{Introduction}\label{introduction}
In the present manuscript we introduce Abelian distributions. 
We have called the distribution Abelian because of a number of identities that arise 
in analysis and that resemble the Abel identity 
$  (x+y)^n=\sum_{i=0}^n {n \choose i} x (x-iz)^{i-1}(y+iz)^{n-i}$~\cite{Saslaw1989}. 
This distribution appeared in 2002
in the study of a fully connected neural network~\cite{Eurich2002} as a distribution of sizes of ``avalanches'' of 
neural activity. Apart from Ref.~\cite{LevinaDiss}, so far there
no systematic and accessible study of the distribution has been published. 
The related results that were reported in the context of Cayley's theorem~\cite{Denker2011} are also based on Ref.~\cite{LevinaDiss}. 
Here we will discuss the basic properties of this probability mass distribution and describe its importance for the applications
in theoretical physics and biology.

\section{Definition}

\begin{definition}
Let $N\in\mathbb{N}$, $\alpha\in(0,1)$. The Abelian distribution is defined for $0\leq L\leq N$ by
\begin{equation} \label{dist_def1}
P_{\alpha,N}(L) = C_{\alpha,N} {{N\choose L}}\left(L\frac{\alpha}{N}\right)^{L-1} \left(1-L\frac{\alpha}{N}\right)^{N-L-1},
\end{equation}
where 
\begin{equation}
C_{\alpha,N}= \frac{1-\alpha}{N-(N-1)\alpha} \label{norm_const},
\end{equation} 
is a normalizing constant.
\end{definition}
Because $P_{\alpha,N}(0) =0$ we will in the following often assume that $L>0$.

\begin{lemma} \label{dist_def2}
The Abelian distribution defined by (\ref{dist_def1}),(\ref{norm_const}) is a probability distribution.
\end{lemma}
\begin{proof}
 We have to show that
 \begin{equation}
\sum_{L=1}^N C_{\alpha,N}{{N \choose L}}\left(L\frac{\alpha}{N}\right)^{L-1} \left(1-L\frac{\alpha}{N}\right)^{N-L-1}=1. 
\nonumber
 \end{equation}
Introducing a new continuous variable $x$ instead of $\alpha/N$, we get
 \begin{equation}
\sum_{L=1}^N {{N \choose L}}\left(L x\right)^{L-1} \left(1-Lx\right)^{N-L-1}=\frac{1}{C_{\alpha,N}}, 
\nonumber
 \end{equation}
which is equivalent to 
 \begin{equation}
\sum_{L=1}^{N-1} {{N \choose L}}\left(L x\right)^{L-1} \left(1-L x\right)^{N-L-1}=
\frac{1}{C_{\alpha,N}}-\frac{(Nx)^{N-1}}{1-Nx}. \label{c3}
 \end{equation}
We can expand the sum on the left side of (\ref{c3}) and obtain
 \begin{equation}
\sum_{L=1}^{N-1}{{N \choose L}}\left(L x\right)^{L-1}
\sum_{m=0}^{N-L-1}(-1)^m {N-L-1 \choose m}(Lx)^m. \label{c4} 
 \end{equation} 
Introducing $k=L$ we can rewrite the sum in the previous expression as a polynomial in $x$
\begin{equation*}
\sum_{i=0}^{N-2}x^i\sum_{k=0}^{i} (-1)^{i-k}{{N \choose k}}{N-k-1 \choose i-k}(k)^{i}=\sum_{i=0}^{N-2}P_i(N) x^i,
\end{equation*}
where $P_i(N)$ is a polynomial in $N$ of degree $i$. For every $N$ we have $P_0(N)=1$. Consider now $i>0$. To identify uniquely 
the polynomial $P_i(N)$ it is sufficient to find its values in $i+1$ different points that we select to be $N=1,\ldots,i+1$.
Because ${{N-1 \choose k}}=0$ for $k>N-1$, we have also ${N-k-2 \choose i- k}=0$ for $N<i+2$ for any $k<N-1$.
Hence, 
 \begin{equation*}
P_i(N)=(-1)^{i-k}{{N \choose N}}{-1 \choose i- k}N^{i} =N^{i} \; \mathrm{for }\; N=1,\ldots,i+1 \;\mathrm{and}\;i>0.
\end{equation*} 
This means that $P_i(N)=N^{i-1} $ for any $N$ and $i>0$. Therefore the left side of (\ref{c4}) is 
 \begin{equation}
1+\sum_{i=1}^{N-2}x^i N^{i-1}=1+x\frac{1-(Nx)^{N-2}}{1-Nx}. \label{c5}
\end{equation}
Inserting (\ref{c5}) and (\ref{norm_const}) into (\ref{c3}) we arrive at
 \begin{equation*}
1+x\frac{1-(Nx)^{N-2}}{1-Nx}=\frac{N-(N-1)\alpha}{N(1-\alpha)}-\frac{(Nx)^{N-1}}{N(1-Nx)},
\end{equation*}
which holds for any $N$ and $\alpha<1$. 
\end{proof}
The authors of Ref.~\cite{Denker2011} mention that the theorem can also be proved by using a generalized binomial theorem.

\begin{figure}[htbp]
  \begin{center}
    \includegraphics[width=0.8\textwidth]{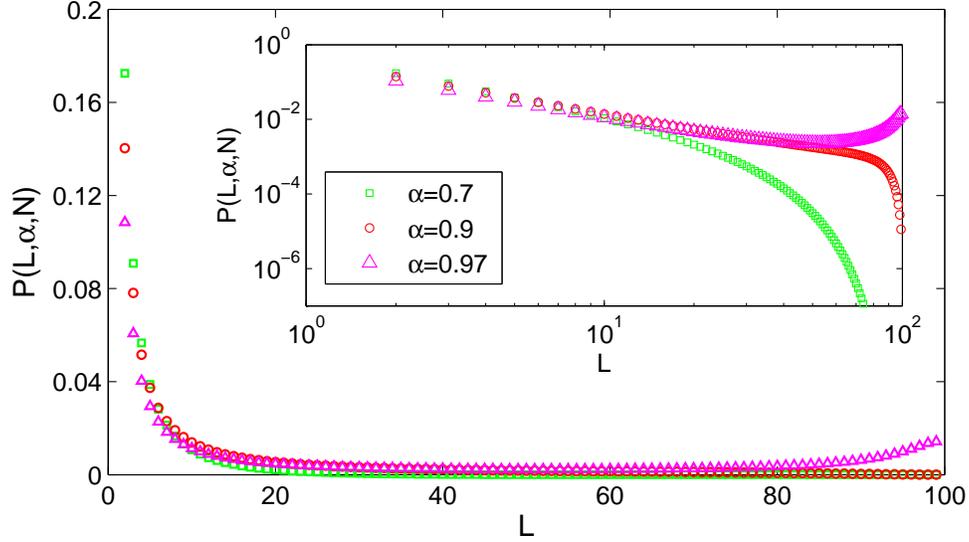} 
  \end{center}
\vspace*{8pt}
\caption{A probability mass function obeying the Abelian distribution for $N=100$ and 
several values of the parameter $\alpha$. Scales are linear (main figure) and double logarithmic (inset).
}\label{Fig:prob_mass}
\end{figure}

An Abelian-distributed probability mass function is shown in Fig.~\ref{Fig:prob_mass} for several values of the parameter $\alpha$. 
For small values of parameter $\alpha<0.9$ 
distribution is monotone and is dominated by approximately exponential decay, for $\alpha \lessapprox 1$ distribution is 
non-monotonous. For some small interval of parameter values  $\alpha \approx 0.9$ the distribution closely resembles 
a power-law (with exponential cutoff at large $L$), see the double logarithmic plot in the inset.
If a sample of data-points of size $10^5$ is drawn from this distribution, the hypothesis of an underlying
power-law distribution cannot be rejected~\cite{LevinaDiss}. 

The shape of the distribution varies in a similar way for all $N$, although for large $N$ the non-monotonous regime is 
present only for $\alpha\in\left(\alpha_{\mathrm{crit}}(N),1\right)$ where the value of $\alpha_{\mathrm{crit}}(N)$ 
has been numerically found to behave roughly as $1-\frac{1}{\sqrt{N}}$.

\section{Expected value}

We will now consider the moments of the Abelian distribution.

\begin{theorem}\label{Th:expectation}
Suppose $\xi$ has an Abelian distribution with parameters $\alpha$ and $N$, then 
\begin{equation*} 
E\xi =\frac{N}{N-(N-1)\alpha}.
\end{equation*}
\end{theorem}
\begin{proof} From (\ref{dist_def1}) and Lemma~\ref{dist_def2} we have
 \begin{equation*}
E\xi =\sum_{L=1}^N L^{L-1}{{N-1 \choose L-1}}\left(\frac{\alpha}{N}\right)^{L-1} \left(1-L\frac{\alpha}{N}\right)^{N-L-1}\frac{N(1-\alpha)}{N-(N-1)\alpha}.
\end{equation*}
We have to prove that
 \begin{equation}
\sum_{L=1}^N L^{L-1}{{N-1 \choose L-1}}\left(\frac{\alpha}{N}\right)^{L-1} \left(1-L\frac{\alpha}{N}\right)^{N-L-1}=\frac{1}{1-\alpha}. 
\nonumber
\end{equation}
Using again $x=\alpha/N$ we can rewrite this equation as
 \begin{equation}
\sum_{L=1}^N{{N-1 \choose L-1}}(Lx)^{L-1} \left(1-Lx \right)^{N-L-1}=\frac{1}{1-N x}. \label{2prove2}
\end{equation}
Transforming the sum in (\ref{2prove2}) we obtain 
\begin{equation*}
\sum_{L=1}^{N-1}{{N-1 \choose L-1}}(Lx)^{L-1} \left(1-Lx \right)^{N-L-1}+(Nx)^{N-1}(1-Nx)^{-1}=\frac{1}{1-Nx}. 
\end{equation*}
which is equivalent to
 \begin{equation}
\sum_{L=1}^{N-1}{{N-1 \choose L-1}}(Lx)^{L-1} \left(1-Lx \right)^{N-L-1}=\sum_{i=0}^{N-2}(Nx)^{i}. \label{2prove3} 
\end{equation} 
Both the left and the right side of the equation (\ref{2prove3}) are polynomials in $x$ of degree $N-2$. Hence in order to prove 
that equation (\ref{2prove3}) is an identity it is sufficient to show that the coefficients of $x^i$ on the both sides are equal
for every $i$. In other words, we have to show that
 \begin{equation}
\sum_{k=0}^{i}(-1)^{i-k}{{N-1 \choose k}}{N-k-2 \choose i- k}(k+1)^{i} =N^{i}. \label{2prove4} 
\end{equation} 
Again, both sides of (\ref{2prove4}) are polynomials of $N$ of the degree $i$. It is sufficient to prove that both sides of (\ref{2prove4}) are equal for $i+1$ different points. We can select these points to be $N=1, \ldots, i+1$. 

Obviously, if $k>N-1$, then ${{N-1 \choose k}}=0$, but also ${N-k-2 \choose i- k}=0$ for $k<N-1$ because $N<i+2$. 
Hence the only non-zero item of the sum is the one corresponding to $k=N-1$, in this case we have 
 \begin{equation*}
(-1)^{i-k}{{N-1 \choose N-1}}{-1 \choose i- k}N^{i} =N^{i}.
\end{equation*} 
\end{proof}

\section{Motivation}

Power-law distributions have been studied in the sciences for a long time, the most prominent example 
being the Gutenberg-Richter law which describes the energy distribution in earthquakes~\cite{Gutenberg1954}. 
Other examples~\cite{Bakb} include forest fires, migratory patterns, infectious diseases, 
solar flares, sandpiles~\cite{Bak1987} and 
neural activity dynamics~\cite{Eurich2002,Beggs2003,Levina2007,Levina2009}.
Some of these examples can be related to critical branching processes~\cite{kolmogorov1947branching} which 
are known to produce power-law event distributions~\cite{Otter1949}.
The relation between power-laws and branching processes usually requires a limit of large systems size~\cite{Levina2008}
which is, however, not relevant when a comparison to numerical computations or mesoscopic experiments is desired.
Nevertheless, the Abelian distribution converges to a power-law (asymptotically for large event sizes 
$L\to \infty$ or as an event density) in the 
exchangeable limits $N\to \infty$ and $\alpha\to 1$. The exponent of the power-law $\gamma = - \frac{3}{2}$ 
is closely obeyed even for small $L$.
Criticality being defined as the divergence of certain physical quantities (such as the mean event size) 
cannot occur in finite systems.  Therefore it is tempting to use the Abelian distribution to define an 
analogon of criticality also for finite systems. 
Depending on the parameters the Abelian distribution has monotonic or non-monotonic behavior, the latter being characterized 
by a relative dominance of events with a size near the size of the system. The two behaviors, the sub- and the supercritical 
regime are separated by a ``critical'' distribution, which is, however, unambiguously defined in terms of a power-law
only for large systems. Avoiding the dependence of the critical parameters on the sample size that may arise when using a 
test (e.g.~Kolmogorov-Smirnov) in order to determine the likelihood of criticality, we propose instead to define criticality 
by qualitative criteria implied by the local similarity to a power-law. Consider the set $A(N)$ of parameters $\alpha$ for 
which the equation $\frac{d^2 \log P_{\alpha,N}(L)}{d(\log L)^2}=0$ has a solution $L\in\{1,\dots,N\}$ for fixed $N<\infty$. 
Expecting $A(N)$ to contract into $\{1\}$ for $N\to \infty$, we can define $A(N)$ as the critical region for a finite system.
Another possibility is to define a single critical value $\alpha_{\mathrm{crit}}$ as an $\sup_{\alpha<1}\left\{\alpha: \frac{d \log P_{\alpha,N}(L)}{d L}<0, \forall L<N \right\}$. This definition uses the property of a critical state to stay between strictly monotonous and non-monotonous regimes.
For all our numerical evaluations we found $\alpha_{\mathrm{crit}}\in A(N)$. 
Thus, the Abelian distribution is one of the few cases where 
the emergence of criticality in an infinite system can be studied explicitly as a 
limit of finite systems which enables a direct comparison with numerical computations or
mesoscopic experiments. 

The Abelian distribution has been studied mainly in the context of neural avalanche dynamics~\cite{Eurich2002,LevinaDiss}, where
it not only turned out be successful in predicting an experimental result from  neuroscience~\cite{Beggs2003}, 
but also allowed for an explicit and exact study of finite size effects. It is interpreted in this context as 
the conditional probability of $L-1$ other neurons being activated given that one neuron just became spontaneously 
active, thus forming an avalanche of $L$ neural action potentials.
From Theorem~\ref{Th:expectation} follows that the expectation exists also in the limit of large $N$ if 
$\alpha<1$ as required by Definition~\ref{dist_def1}. Correspondingly, in the neural system, 
a single nonterminating avalanche is observed at $\alpha=1$. 

The application of the Abelian distribution as an event size distribution may require an appropriate definition
of events. Although neurons produce quasi-discrete action potentials, in the experiments~\cite{Beggs2003} events have been defined
by threshold crossings, where an invariance of the distribution of the choice of the threshold is required for justification.
In other time series, events can be defined either in a similar way. While the parameter $N$ has usually a natural interpretation
as the size of the system, for e.g.~financial time series its meaning is less obvious. If $N$ can be found by maximum likelihood, it
can be interpreted as an effective system size. The parameter $\alpha$ describes in all cases the strength of the interaction
between the elements in the system. If the elements are not all connected or if the system is heterogeneous, it seems reasonable to 
use, respectively, connectivity-rescaled parameters or an average interaction strength to determine estimates of this parameter.

\section{Open questions}
A large number of questions related to the Abelian distribution are left for future investigation. Most important among them are the higher moments,
characteristic function, stability and properties related to parameter estimation. Especially interesting for the application to critical system would be
a scaling law for the critical value $\alpha_{\mathrm{crit}}$ and relation between different possibilities to define ctiticality for finite system. 

\section*{Acknowledgments}
The authors wish to thank Zakhar Kablutschko and Theo Geisel for helpful discussions and 
Manfred Denker for valuable comments, help and support.
Supported by the Federal Ministry of Education and Research (BMBF) Germany under grant number 01GQ1005B.

\end{document}